\documentclass{amsart}[12pt]
\usepackage{amsmath, amsthm, amscd, amsfonts,}
\usepackage{graphicx,float}

\usepackage{amssymb}

\usepackage{pb-diagram}

\DeclareMathOperator{\dom}{dom}

\newtheorem{theorem}{Theorem}[section]
\newtheorem{lemma}[theorem]{Lemma}

\newtheorem{remark}[theorem]{Remark}

\newtheorem{question}[theorem]{Question}
\numberwithin{equation}{section}

\usepackage{pb-diagram}

\def\rmark{\mbox{$\rm\bf\rule{0.06em}{1.45ex}\kern-0.05em R$}}
\def\pmark{\mbox{$\rm\bf\rule{0.06em}{1.45ex}\kern-0.05em P$}}
\def\nmark{\mbox{$\rm\bf\rule{0.06em}{1.45ex}\kern-0.05em N$}}
\def\vdash{\mbox{$\rm\| \kern-0.13em -$}}
\newcommand{\lusim}[1]{\smash{\underset{\raisebox{1.2pt}[0cm][0cm]{$\sim$}}
{{#1}}}}

\begin{document}

\title[More on almost Souslin Kurepa trees]{More on almost Souslin Kurepa trees}

\author[M. Golshani]{Mohammad
 Golshani}

\thanks{This research was in part supported by a grant from IPM (No. 91030417). The author also would like to thank the referee of the paper for some useful remarks and comments.} \maketitle




\begin{abstract}
It is consistent that there exists a Souslin tree $T$ such that after forcing with it, $T$ becomes an almost Souslin Kurepa tree. This answers a question of Zakrzewski \cite{z}.
\end{abstract}

\maketitle


\section{Introduction}

In this paper we continue our  study of $\omega_1$-trees started in \cite{gol} and
prove another consistency result concerning them. Let $T$ be a
normal $\omega_1$-tree. Let's recall that:

\begin{itemize}
 \item $T$
is a Kurepa tree if it has at least $\omega_2$-many branches.
\item $T$ is a Souslin tree if it has no uncountable antichains
(and hence no branches). \item $T$ is an almost Souslin tree if
for any antichain $X \subseteq T,$ the set $S_X = \{ht(x):x \in X
\}$ is not stationary (see \cite{ds}, \cite{z}).
\end{itemize}

We refer to \cite{gol} and \cite{j1} for  historical information and more details on trees.

In \cite{z}, Zakrzewski asked some questions concerning the existence of almost Souslin Kurepa trees. In \cite{gol} we answered two of these questions but one of them remained open:

\begin{question}
Does there exist a Souslin tree T such that for each
$G$ which is $T$-generic over $V, T$ is an almost Souslin
Kurepa tree in $V[G]$?
\end{question}

In this paper we give an affirmative answer to this question.
\begin{theorem}
It is consistent that there exists a Souslin tree T such that for each
$G$ which is $T$-generic over $V, T$ is an almost Souslin
Kurepa tree in $V[G].$
\end{theorem}
The rest of this paper is devoted to the proof of this theorem. Our proof is motivated by \cite{f} and \cite{gol}.
\section{Proof of Theorem 1.2}
Let $V$ be a model of $ZFC+GCH.$ Working in $V$ we define a forcing notion which adds a Souslin tree which is almost Kurepa, in the sense that $T$ becomes a Kurepa tree in its generic extension. The forcing notion is essentially the forcing notion introduced in \cite{f} and we will recall it here for our later usage.
Conditions $p$ in $\mathbb{S}$ are of the form $\langle t, \langle \pi_{\alpha}: \alpha \in I  \rangle   \rangle,$ where we write $t=t_p, I=I_p$ and $\langle \pi_{\alpha}: \alpha \in I  \rangle=\vec{\pi}^p$ such that:
\begin{enumerate}
\item $t$ is a normal $\omega$-splitting tree of countable height $\eta$, where $\eta$ is either a limit of limit ordinals or the successor of a limit ordinal. We denote $\eta$ by $\eta_p.$
\item $I$ is a countable subset of $\omega_2.$
\item Every $\pi_{\alpha}$ is an automorphism of $t\upharpoonright Lim,$ where $Lim$ is the set of countable limit ordinals and $t\upharpoonright Lim$ is obtained from $t$ by restricting its levels to $Lim.$

\end{enumerate}
The ordering is the natural one: $\langle s, \vec{\sigma} \rangle \leq  \langle t, \vec{\pi}   \rangle$ iff $s$ end extends $t, dom(\vec{\sigma}) \supseteq dom(\vec{\pi})$ and for all $\alpha \in dom(\vec{\pi}), \sigma_{\alpha} \upharpoonright t=\pi_{\alpha}.$
\begin{remark}
In \cite{f}, the conditions in $\mathbb{S}$ must satisfy an additional requirement that we do not impose here. This is needed in \cite{f} to ensure the generic $T$ is rigid. Its exclusion does not affect our proof, and in fact simplifies several details.
\end{remark}

Let
 \begin{center}
$\mathbb{P}=\{ p\in \mathbb{S}:$ for some $\alpha_p,  \eta_p =\alpha_p +1   \}.$
\end{center}
It is easily seen that $\mathbb{P}$ is dense in $\mathbb{S}.$
Let $G$ be $\mathbb{P}$-generic over $V$. Let
\begin{center}
$T=\bigcup \{t_p: p\in G    \}$
\end{center}
and for each $\alpha< \omega_2$ set
\begin{center}
$\pi_i= \bigcup \{\sigma_i: \exists u=\langle t, \vec{\sigma} \rangle \in G, i \in I_u   \}.$
\end{center}

 Then (see \cite{f}, Lemmas 2.3, 2.7, 2.9 and 2.14):
\begin{lemma}

$(a)$ $\mathbb{P}$ is $\omega_1$-closed and satisfies the $\omega_2$-$c.c.,$

$(b)$  $T= \langle \omega_1, <_T  \rangle$ is a Souslin tree.

$(c)$ Each $\pi_i$ is an automorphism of $T\upharpoonright Lim.$

$(d)$ If $b$ is a branch of $T,$ which is $T$-generic over $V[G],$ and if $b_i=\pi_{i}``b, i< \omega_2,$ then the $b_i$'s are distinct branches of $T$. In particular $T$ is almost Kurepa.
\end{lemma}
Let
$S=\{\alpha_p: p\in G, \alpha_p=\bigcup \{\alpha_q: q \in G, \alpha_q < \alpha_p  \}$ and $I_p=\bigcup \{I_q: q \in G, \alpha_q < \alpha_p  \}    \}.$
Then as in \cite{gol}, Lemma 2.4, we can prove the following:
\begin{lemma}
$S$ is a stationary subset of $\omega_1.$
\end{lemma}
Working in $V[G]$ let $\mathbb{Q}$ be the usual forcing notion for adding a club subset of $S$ using countable conditions and let $H$ be $\mathbb{Q}$-generic over $V[G].$ Then (see \cite{j2} Theorem 23.8):
\begin{lemma}
$(a)$ $\mathbb{Q}$ is $\omega_1$-distributive and satisfies the $\omega_2$-$c.c.,$

$(b)$ $C=\bigcup H \subseteq S$ is a club subset of $\omega_1.$

\end{lemma}
Let
\begin{center}
$\mathbb{R}=\{\langle p, \check{c}  \rangle: p \in \mathbb{P}, p\vdash  \check{c} \in \lusim{\mathbb{Q}}$ and $\max(c)\leq \alpha_p \}.$
\end{center}
Since $\mathbb{P}$ is $\omega_1$-closed, $\mathbb{Q} \subseteq V$ and hence we can easily show that $\mathbb{R}$ is dense in $\mathbb{P}*\lusim{\mathbb{Q}}.$
\begin{lemma}
$T$ remains a Souslin tree in $V[G][H].$
\end{lemma}
\begin{proof}
We work with $\mathbb{R}$ instead of $\mathbb{P}*\lusim{\mathbb{Q}}.$ Let $\lusim{A}$ be an $\mathbb{R}$-name,  $r_0 \in \mathbb{R}$ and $r_0 \vdash$``$\lusim{A}$ is a maximal antichain in $\lusim{T}$''. Let $\lusim{f}$ be a name for a function that maps each countable ordinal $\alpha$ to the smallest ordinal in $\lusim{A}[G*H]$ compatible with $\alpha$. Then as in \cite{f} we can define a decreasing sequence $\langle r_n: n< \omega  \rangle$ of conditions in $\mathbb{R}$ such that
\begin{itemize}
\item $r_0$ is as defined above,
\item $r_n=\langle p_n, \check{c}_n \rangle = \langle \langle t_n, \vec{\pi}^n  \rangle, \check{c}_n  \rangle,$
\item $\alpha_{p_n}< \alpha_{p_{n+1}},$
\item $r_{n+1}$ decides $\lusim{f} \upharpoonright t_{n},$ say it forces ``$\lusim{f} \upharpoonright t_{n}=\check{f_n}$'',
\item $r_{n+1} \vdash \lusim{C} \cap (\alpha_{p_n}, \alpha_{p_{n+1}}) \neq \emptyset,$

\end{itemize}
Let  $p= \langle t, \vec{\pi}  \rangle$ where $t= \bigcup_{n<\omega}t_{n}, \dom(\vec{\pi})= \bigcup_{n<\omega}\dom(\vec{\pi}^n)$ and for $i\in \dom(\vec{\pi}), \pi_i=\bigcup_{n<\omega}\pi_i^n.$ Let $c=\bigcup_{n<\omega}c_n \cup \{\alpha_p \},$ where $\alpha_p=\sup_{n<\omega}\alpha_{p_n}.$ Then $p \in \mathbb{S},$ but it is not clear that $p \vdash$``$ \check{c}\in\lusim{\mathbb{Q}}$''.

Let $f=\bigcup_{n<\omega}f_n$ and set $a=ran(f\upharpoonright t).$ As in \cite{f}, Lemma 2.9, we can  define a condition $s=\langle q, \check{c} \rangle$ such that
\begin{itemize}
\item $s \in \mathbb{R},$
\item $\eta_q=\alpha_p+1,$ (and hence $\alpha_q=\alpha_p$),
\item $s \vdash$``$\lusim{A} \cap \check{t}$ is a maximal antichain in $\check{t}$'',

\item Every new node  (i.e. every node at the $\alpha_p$-th level) of the tree part of $s$ is above a condition in $a.$
\end{itemize}

It is now clear that $s \vdash \lusim{A}=\check{a},$ and hence $s \vdash$``$ \lusim{A}$ is countable''. The lemma follows.
\end{proof}
From now on we work in $V^*=V[G][H].$ Thus in $V^*$ we have a Souslin tree $T.$ We claim that $T$ is as required. To see this force with $T$ over $V^*$ and let $b$ be a branch of $T$ which is $T$-generic over $V^*$.
\begin{lemma}
In $V^*[b], T$ is an almost Souslin Kurepa tree.
\end{lemma}
\begin{proof}
Work in $V^*[b].$ By Lemma 2.2$(d)$ $T$ is a Kurepa tree.
We now show that $T$ is almost Souslin. We may suppose that $T$ is obtained using the branches $b$ and $b_i, i<\omega_2,$ in the sense that for each $\alpha<\omega_1, T_{\alpha},$ the $\alpha$-th level of $T,$ is equal to $\{b(\alpha)\}\cup \{ b_{i}(\alpha): i<\omega_2 \}$ where $b(\alpha)$ ($b_i(\alpha)$) is the unique node in $b\cap T_{\alpha}$ ($b_i \cap T_{\alpha}$). We further suppose that $b=b_0$.

Now let $\alpha \in C,$ and let $p\in G$ be such that $\alpha=\alpha_p.$ We define a function $g_{\alpha}$ on $T_{\alpha}$ as follows. Note that $T_{\alpha}=\{b_i(\alpha): i\in I_p \}.$ Let
\begin{center}

$g_{\alpha}(b_i(\alpha))=b_i(\alpha_q)$
\end{center}
where $q\in G$ is such that $\alpha_q< \alpha$ is the least such that $i\in I_q$ (such a $q$ exists using the fact that $C \subseteq S$).
It is easily seen that $g_{\alpha}$ is well-defined (it does not depend on the choice of $p$), and that for each $x\in T_{\alpha}, g_{\alpha}(x)<_T x.$
The rest of the  proof of the fact that $T$ is almost Souslin is essentially the same as in \cite{gol}, Lemma 2.6.
\end{proof}
This concludes the proof of Theorem 1.2.

School of Mathematics, Institute for Research in
Fundamental Sciences (IPM), P.O. Box: 19395-5746, Tehran-Iran.

e-mail: golshani.m@gmail.com

\end{document}